\documentclass[12pt]{article}
\usepackage{fullpage}
\usepackage{verbatim}
\usepackage{amsthm}
\usepackage{latexsym}
\usepackage{graphics,epsf,psfrag,epstopdf,url}

\newtheorem{theorem}{Theorem}
\newtheorem{lemma}[theorem]{Lemma}
\newtheorem{corollary}[theorem]{Corollary}
\newtheorem{conjecture}{Conjecture}
\newtheorem{question}[conjecture]{Question}
\newtheorem*{oneshot1}{Theorem 1}
\newtheorem*{oneshot}{Theorem}
\newtheorem*{oneshotC}{Conjecture}
\def\floor#1{\lfloor #1 \rfloor}
\def\cyc{C}
\def\Ad{\textrm{ad}}
\def\mad{\textrm{mad}}
\def\Mad{\textrm{mad}}
\def\ex{\textrm{excess}}
\newcommand\dist[1]{\mbox{distance$#1$}}
\begin{document}
\author{Daniel W. Cranston\thanks{dcransto@uiuc.edu} \\
University of Illinois \\
Urbana-Champaign, USA
\and Seog-Jin Kim\thanks{skim12@konkuk.ac.kr} \\
Konkuk University \\
Seoul, Korea}
\title{List-coloring the Square of a Subcubic Graph}
\date{August 8, 2007}
\maketitle
\abstract{
The {\em square} $G^2$ of a graph $G$ is the graph with the same vertex set
as $G$ and with two vertices adjacent if their distance in $G$ is at most 2. 
Thomassen showed that every planar graph $G$ with maximum degree $\Delta(G)=3$
satisfies $\chi(G^2)\leq 7$.  Kostochka and Woodall conjectured that for every graph, the list-chromatic number of $G^2$ equals the chromatic number of $G^2$, that is $\chi_l(G^2)=\chi(G^2)$ for all $G$.  If true, this conjecture (together with Thomassen's result) implies that every planar graph $G$ with $\Delta(G)=3$ satisfies $\chi_l(G^2)\leq 7$.
We prove that every connected graph (not necessarily planar) with $\Delta(G)=3$ other than the Petersen graph satisfies $\chi_l(G^2)\leq 8$ (and this is best possible).  In addition, we show that if $G$ is a planar graph with $\Delta(G)=3$ and girth $g(G)\geq 7$, then $\chi_l(G^2)\leq 7$.
Dvo\v{r}\'{a}k, \v{S}krekovski, and Tancer showed that if $G$ is a planar graph with $\Delta(G) = 3$ and girth $g(G) \geq 10$, then $\chi_l(G^2)\leq 6$.
We improve the girth bound to show that if $G$ is a planar graph with $\Delta(G)=3$ and $g(G) \geq 9$, then $\chi_l(G^2) \leq 6$.
All of our proofs can be easily translated into linear-time coloring algorithms.
}

\section{Introduction}
We study the problem of coloring the square of a graph.
We consider simple undirected graphs.
Since each component of a graph can be colored independently, we only consider connected graphs.
The \textit{square} of a graph $G$, denoted $G^2$, has the same vertex set as $G$ and has an edge between two vertices 
if the distance between them in $G$ is at most 2.  
We use $\chi(G)$ to denote the chromatic number of $G$.  We use $\Delta(G)$ to denote the largest degree in $G$.  We say a graph
$G$ is \textit{subcubic} if $\Delta(G)\leq 3$.

Wegner \cite{Wegner} initiated the study of the chromatic number for
squares of planar graphs.  This topic has been actively studied
lately due to his conjecture. 

\begin{oneshotC}(Wegner \cite{Wegner}) 
Let $G$
be a planar graph.  The chromatic number $\chi(G^2)$
of $G^2$ is at most 7 if $\Delta(G) = 3$, at most $\Delta(G) + 5$ if $4
\leq \Delta(G) \leq 7$, and at most $\lfloor \frac{3 \Delta(G)}{2} \rfloor
+1$ otherwise. 
\end{oneshotC} 

Thomassen \cite{Thomassen} proved Wegner's
conjecture for $\Delta(G) =3$, but it is still open for all values of
$\Delta(G) \geq 4$. The best known upper bounds are due to Molloy and
Salavatipour \cite{MS}, who showed that $\chi(G^2)\leq\lceil\frac53\Delta\rceil+78$ (the constant 78 can be reduced for large $\Delta$). 
Very recently, Havet, Heuvel, McDiarmid, and Reed~\cite{havet2} proved the upper bound $\frac32\Delta(1+o(1))$.
Better results can be obtained for special
classes of planar graphs. Borodin {\em et al.} \cite{borodin} and
Dvo\v{r}\'{a}k {\em et al.} \cite{kral} proved that $\chi(G^2)=
\Delta(G) +1 $ if $G$ is a planar graph $G$ with sufficiently large
maximum degree and girth at least 7.
A natural variation of this problem is to study the list chromatic number
of the square of a planar graph.

A {\em list
assignment} for a graph is a function $L$ that assigns each vertex a
list of available colors.  The graph is $L$-{\em colorable} if it
has a proper coloring $f$ such that $f(v) \in L(v)$ for all $v$. A
graph is called {\em $k$-choosable} if $G$ is $L$-colorable whenever
all lists have size $k$. The list chromatic number
$\chi_l(G)$ is the minimum $k$ such that $G$ is $k$-choosable.
Kostochka and Woodall~\cite{kostochka} conjectured that $\chi_l(G^2) = \chi(G^2)$ 
for every graph $G$.

We consider the problem of list-coloring $G^2$ when $G$ is subcubic.
If $G$ is subcubic then clearly $\Delta(G^2)\leq (\Delta(G))^2\leq 9$.
It is an easy exercise to show that the Petersen graph is the only subcubic graph $G$ such that $G^2=K_{10}$.
Hence, by the list-coloring version of Brooks' Theorem~\cite{ERT} we conclude that if $G$ is subcubic and $G$ is not the
Petersen graph, then $\chi_l(G^2)\leq\Delta(G^2)\leq 9$.  In fact, we show that this upper bound can be strengthened as follows.
We say that a subcubic graph is \textit{non-Petersen} if it is not the Petersen graph.

\begin{theorem}
\label{mainthm}
If $G$ is a non-Petersen subcubic graph, then $\chi_l(G^2)\leq 8$.
\end{theorem}

\begin{figure}[!h!bt]
  \begin{center}
    \psfrag{V1}{{\normalsize$v_1$}}
    \psfrag{V2}{{\normalsize$v_2$}}
    \psfrag{V3}{{\normalsize$v_3$}}
    \psfrag{V4}{{\normalsize$v_4$}}
    \psfrag{V5}{{\normalsize$v_5$}}
    \psfrag{V6}{{\normalsize$v_6$}}
    \psfrag{V7}{{\normalsize$v_7$}}
    \psfrag{V8}{{\normalsize$v_8$}}
    \psfrag{U1}{{\normalsize$v_1$}}
    \psfrag{U2}{{\normalsize$v_2$}}
    \psfrag{U3}{{\normalsize$v_3$}}
    \psfrag{U4}{{\normalsize$v_4$}}
    \psfrag{U5}{{\normalsize$v_5$}}
    \psfrag{U6}{{\normalsize$v_6$}}
    \psfrag{U7}{{\normalsize$v_7$}}
    \psfrag{U8}{{\normalsize$v_8$}}
    \psfrag{First}{(a)}
    \psfrag{Second}{(b)}
\scalebox{.88}{
    \includegraphics{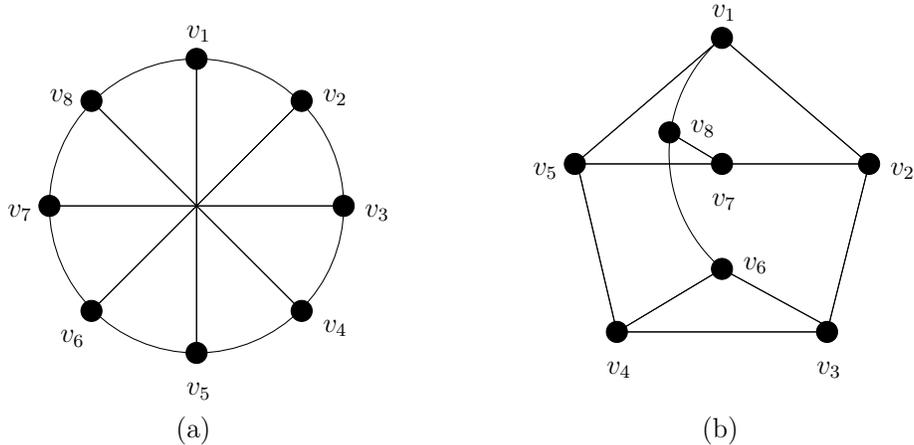}
}
  \end{center}
\caption{Two graphs, each on 8 vertices; each has $K_8$ as its square. (a) An 8-cycle
$v_1,v_2,\ldots,v_8$ with ``diagonals'' (i.e. the additional edges are $v_iv_{i+4}$ for each $i\in\{1,2,3,4\}$).  This graph has girth 4.
(b) This graph has girth 3.
}
\end{figure}

Theorem~\ref{mainthm} is best possible, as illustrated by the graphs above.
The graph on the left has girth 4.  The graph on the right has girth 3.
The square of each graph is $K_8$.  Thus, each graph requires lists of size 8.
In fact, there are an infinite number of interesting subcubic graphs $G$ such 
that $\chi_l(G^2)=8$.  Let $H$ be the Petersen graph with an edge removed.
Note that $H^2 = K_8$.  Hence, any graph $G$ which contains $H$ as a subgraph
satisfies $\chi_l(G^2)\geq 8$.

In Section~\ref{prelims} we introduce definitions and themes common to our proofs.
In Section~\ref{mainsec} we prove Theorem~\ref{mainthm}.
In Section~\ref{planargirth7} we show that if $G$ is a planar subcubic graph with girth at least 7, then $\chi_l(G^2)\leq 7$.
Dvo\v{r}\'{a}k, \v{S}krekovski, and Tancer \cite{DST} showed that 
if $G$ is a subcubic planar graph
with girth at least 10, 
then $\chi_l(G^2)\leq 6$.
In Section~\ref{planargirth9} we extend their result by lowering the girth bound from 10 to 9.

\section{Preliminaries}
\label{prelims}
We use $n$, $e$, and $f$ to denote the number of vertices, edges, and faces in a graph.
A \textit{partial (proper) coloring} is the same as a proper coloring except that some vertices may be uncolored.
We use $g(G)$ to denote the girth of graph $G$.  When the context is clear, we simply write $g$.
We use $k$-vertex to denote a vertex of degree $k$.
We use $\textit{ad}(G)$ to denote the average degree of a graph.
Similarly, we use $\textit{mad}(G)$ to denote the maximum average degree of $G$; that is, the
maximum of $2|E(H)|/|V(H)|$ over all induced subgraphs $H$ of $G$.
We use $N[v]$ to denote the closed neighborhood of $v$ in $G^2$.
We use $G[V_1]$ to denote the subgraph of $G$ induced by vertex set $V_1$.

Throughout the paper, we use the idea of \textit{saving a color} at a vertex $v$.  By this we mean that we assign colors to two neighbors of $v$ in $G^2$ but we only reduce the list of colors available at vertex $v$ by one.  A typical example of this occurs when $v$ is adjacent to vertices $v_1$ and $v_2$ in $G^2$, $v_1$ is not adjacent to $v_2$ in $G^2$, and $|L(v_1)|+|L(v_2)| > |L(v)|$.  This inequality implies that either $L(v_1)$ and $L(v_2)$ have a common color or that some color appears in $L(v_1)\cup L(v_2)$ but not in $L(v)$.  In the first case, we save by using the same color on vertices $v_1$ and $v_2$.  In the second case, we use a color in $(L(v_1)\cup L(v_2))\setminus L(v)$ on the vertex where it appears and we color the other vertex arbitrarily.

We say a graph $G$ is \textit{$k$-minimal} if $G^2$ is not $k$-choosable, but the square of every proper subgraph of $G$ is
$k$-choosable.
A \textit{configuration} is an induced subgraph.
We say that a configuration is \textit{$k$-reducible} if it cannot appear in a $k$-minimal graph (we will be interested in $k\in\{6,7,8\}$).
We say that a configuration is \textit{$6'$-reducible} if it cannot appear in a 6-minimal graph with girth at least 7.
Note that for every $k\geq 4$ a 1-vertex is $k$-reducible 
(if $G$ contains a 1-vertex $x$, by hypothesis we can color $(G-x)^2$, then we can extend the coloring to $x$ since in this case $(G-x)^2=G^2-x$).  
Hence, in the rest of this paper, we assume our graphs have no 1-vertices.

Note that the definition of $k$-minimal requires that for every subgraph $H$ the square of $G- V(H)$ is $k$-choosable, but does not require the stronger statement that for every subgraph $H$ the graph $G^2- V(H)$ is $k$-choosable.  This is a subtle, but significant distinction.  
To avoid trouble, in Sections 4 and 5 we will only consider reducible configurations $H$ such that $G^2- V(H) = (G- V(H))^2$; otherwise, we may face difficulties as in the next paragraph.  

We give a fallacious proof that $\chi_l(G^2)\leq 7$ for every subcubic planar graph $G$ with girth at least 6.  Clearly, a vertex of degree 2 is a 7-reducible configuration (and so is a vertex of degree 1), since it has degree at most 6\ in $G^2$.  Let $G$ be a 7-minimal subcubic planar graph of girth at least 6.  Since, $G$ is planar and has girth at least 6, $G$ has a vertex $v$ of degree at most 2 (by Lemma~\ref{girthlemma}).  By hypothesis, we can color $G^2-v$.  Since $v$ has at most 6 neighbors in $G^2$ we can extend the coloring to $v$.

The flaw in this proof is that by hypothesis, we can color $(G-v)^2$, which may have one less edge than $G^2-v$; in particular, if $v$ is adjacent to vertices $u$ and $w$, then $G^2-v$ contains the edge $uw$, but $(G-v)^2$ does not.
We may be tempted to add the edge $uw$ to the graph $(G-v)^2$; however, if we do, the new graph may not satisfy the girth restriction.

In both Section~\ref{planargirth7} and Section~\ref{planargirth9} we make use of 
upper bounds on $\Mad(G)$.
To prove these bounds, we use the following well-known lemma.

\begin{lemma}
\label{girthlemma}
If $G$ is a planar graph with girth at least $g$, then $\Mad(G) < \frac{2g}{g-2}$.
\end{lemma}
\begin{proof}
Every subgraph of $G$ is a planar graph with girth at least $g$; hence, it is enough to show that ${\Ad = \frac{2e}n < \frac{2g}{g-2}}$.
From Euler's formula we have $f - e + n = 2$.  By summing the lengths of all the faces, we get $2e\geq fg$.  
Combining these gives the following inequality.  
\begin{eqnarray*}
e &<& e+2 \leq \frac{2e}{\Ad} + \frac{2e}g \\
1 &<& \frac2{\Ad} + \frac2g \\
\Ad &<& \frac{2g}{g-2}
\end{eqnarray*}

\end{proof}
In Section~\ref{mainsec}, we show that given a graph $G$ with lists of size 8, we can greedily color all but a few vertices of $G$, each near a central location.  The ``hard work'' in Section~\ref{mainsec} is showing that we can extend the coloring to these last few uncolored central vertices.

The outlines of Section~\ref{planargirth7} and Section~\ref{planargirth9} are very similar. 
In each section, 
we exhibit four reducible configurations; recall that a reducible
configuration cannot occur in a $k$-minimal graph. 
Next, we show that if a subcubic planar graph with girth at least 7~(resp. 9)
does not contain any of these reducible configurations, then $G$ has
$\Mad(G)\geq \frac{14}5~(\frac{18}7)$. This implies that each
subcubic planar graph of girth 7~(resp. 9) contains a reducible
configuration. It follows that there is no $6'$-minimal subcubic graph with $\mad(G) < \frac{18}7$ and there is no 7-minimal subcubic graph with $\mad(G) < \frac{14}5$; so the theorems are true.

\section{General subcubic graphs}
\label{mainsec}
We begin this section by proving a number of structural lemmas about 8-minimal subcubic graphs.
We conclude by showing that if $G$ is a non-Petersen subcubic graph, 
then $\chi_l(G^2)\leq 8$.
\begin{lemma}
\label{mainlemma}
If $G$ is a subcubic graph, then for any edge $uv$ we have $\chi_l(G^2-\{u,v\})\leq 8$.
\end{lemma}
\begin{proof}
For every vertex $w$ other than $u$ and $v$, we define the \textit{distance class} of $w$ to be the distance in $G$ from $w$ to edge $uv$.
We greedily color the vertices of $G^2-\{u,v\}$ in order of decreasing distance class.  We claim that lists of size 8 suffice.
Note that $|N[w]|\leq 10$ for every vertex $w$.  If at least two vertices in $N[w]$ distinct from $w$ are uncolored when we color $w$, then we need at
most $10-2$ colors at vertex $w$.  Suppose $w$ is in distance class at least 2.  Let $x$ and $y$ be the first two vertices on a 
shortest path in $G$ from $w$ to edge $uv$.  Since vertices $x$ and $y$ are in lower distance classes than $w$, they are both uncolored
when we color $w$.  Hence, we need at most $10-2$ colors at vertex $w$.  If $w$ is in distance class 1, then $u$ and $v$ are
uncolored when we color $w$.  So again we need only $8$ colors.
\end{proof}

\def\excess{\textrm{excess}}
Lemma~\ref{mainlemma} shows that if $G$ is a subcubic graph, then lists of size 8 are sufficient to color all but two adjacent vertices of $G^2$.  Hence, if $H$ is any subgraph that contains an edge, then we can color $G^2- V(H)$ from lists of size 8.
The next lemma relies on the same idea as Lemma~\ref{mainlemma}, but generalizes the context in which it applies.
Given a graph $G$ and a partial coloring of $G^2$, 
we define \textit{excess}($v$) to be $1~+$~(the number of colors available at vertex $v$) $-$ 
(the number of uncolored neighbors of $v$ in $G^2$).  
Note that for any subcubic graph $G$ and any such partial coloring, every vertex $v$ has $\excess(v)\geq 0$.  Intuitively, the excess of a vertex $v$ measures how many colors we have ``saved'' on $v$ (colors are saved either from using the same color on two neighbors of $v$ or simply because $v$ has fewer than 9 neighbors in $G^2$).  For example, if two neighbors of $v$ in $G^2$ use the same color, then $\excess(v)\geq 1$.  Similarly, if $v$ lies on a 4-cycle or a 3-cycle, then $\excess(v)\geq 1$ or $\ex(v)\geq 2$, respectively.  Vertices with positive excess play a special role in finishing a partial coloring.

\begin{lemma}
\label{extlemma}
Let $G$ be a subcubic graph and let $L$ be a list assignment with lists of size 8.
Suppose that $G^2$ has a partial coloring from $L$.  Suppose also that 
vertices $u$ and $v$ are uncolored, are adjacent in $G^2$, and that $\ex(u)\geq 1$ and $\ex(v)\geq 2$.
If we can order the uncolored vertices so that each vertex except $u$ and $v$ is succeeded in the order by at least 2 adjacent vertices in $G^2$, then we can finish the partial coloring.
\end{lemma}
\begin{proof}
We will color the vertices greedily according to the order.
Recall that for each vertex $w$, we have $|N[w]|\leq 10$.  Since at least two vertices in $N[w]$ will be uncolored at the time we color $w$, we will have a color available to use on each vertex $w$ (other than $u$ and $v$).  Since $u$ and $v$ are the only vertices not succeeded by 2 adjacent vertices in $G^2$, they must be the last two vertices in the order.
Because $\ex(u)\geq 1$ and $\ex(v)\geq 2$, we can finish the coloring by greedily coloring $u$, then $v$.
\end{proof}

A simple but useful instance where Lemma~\ref{extlemma} applies is when the uncolored vertices induce a connected subgraph and vertices $u$ and $v$ are adjacent (we order the vertices by decreasing distance (within the subgraph) from edge $uv$).
Whenever we say that we can greedily finish a coloring,
we will be using Lemma~\ref{extlemma}.
Often, we will specify an order for the uncolored vertices; when we do not give an order it is because they induce a connected subgraph. 
The next two lemmas exhibit small configurations which allow us to apply Lemma~\ref{extlemma}.

\begin{lemma}
\label{3reg}
If $G$ is an 8-minimal subcubic graph, then $G$ is 3-regular.
\end{lemma}
\begin{proof}
Suppose $u$ is a vertex with $d(u)\leq 2$.  Let $v$ be a neighbor of $u$.  Note that ${\ex(v)\geq 1}$ and $\ex(u)\geq 3$.
So by Lemma~\ref{extlemma}, we can color $G^2$ from lists of size 8.
\end{proof}

\begin{lemma}
\label{noC3}
If $G$ is an 8-minimal subcubic graph, then $g(G) > 3$.
\end{lemma}
\begin{proof}
Suppose $G$ contains a 3-cycle $uvw$.  
Note that $\ex(u)\geq2$, $\ex(v)\geq2$, and $\ex(w)\geq 2$.
So by Lemma~\ref{extlemma}, we can color $G^2$ from lists of size 8.
\end{proof}

\begin{lemma}
\label{noC4}
If $G$ is an 8-minimal subcubic graph, then $g(G) > 4$.
\end{lemma}
\begin{proof}
Suppose that $G$ is a counterexample.  Let each vertex have a list of size 8.
Observe that if vertex $v$ lies on a 4-cycle, then $\ex(v)\geq 1$.
Note that if $v$ lies on two 4-cycles, then $\ex(v)\geq 2$.
Suppose that $v_1$ lies on two 4-cycles and $v_2$ is adjacent to $v_1$ on some 4-cycle.
Since $\ex(v_2)\geq 1$ and $\ex(v_1)\geq 2$, we can greedily color $G$.
Hence, we assume that no vertex lies on two 4-cycles.
Let $\cyc$ be a 4-cycle in $G$.  Label the vertices of $\cyc$ as $v_1$, $v_2$, $v_3$, $v_4$.
Recall that $G$ is 3-regular (by Lemma~\ref{3reg}).
Let $u_i$ be the neighbor of $v_i$ not on $\cyc$.
We can assume the $u_i$s are distinct, since otherwise either $G$ contains a 3-cycle or some vertex lies on two 4-cycles.
By Lemma~\ref{mainlemma}, we color all vertices except the $u_i$s and $v_i$s. 
Let $L(v)$ denote the list of remaining colors available at each uncolored vertex $v$.

\textit{Case 1:} 
Suppose that $\dist(u_1,v_3)=3$.  
Note that $|L(v_i)|\geq 6$ and $|L(u_i)|\geq 2$.  
We assume that equality holds for $v_1$ (otherwise we discard colors until it does).
Since $|L(u_1)|+|L(v_3)|>|L(v_1)|$, we can choose color $c_1$ for $u_1$ and color $c_2$ for $v_3$
so that $|L(v_1)\setminus\{c_1,c_2\}|\geq 5$.  
Since $\ex(v_2)\geq 1$ and $\ex(v_1)\geq 2$, we can finish the coloring by Lemma~\ref{extlemma} (coloring greedily in the order $u_2$, $u_3$, $u_4$, $v_4$, $v_2$, $v_1$).

\textit{Case 2:} 
Suppose instead that $\dist(u_1,v_3)<3$. 
Vertices $u_1$ and $u_3$ must be adjacent; by symmetry $u_2$ and $u_4$ must be adjacent or we get the result by Case 1.
Now since $u_1$ and $u_3$ are adjacent and $u_2$ and $u_4$ are adjacent (see Figure 2),
we have $|L(v_i)|\geq 7$ and $|L(u_i)|\geq 4$ (we assume that equality holds for the $v_i$s).
Suppose that $\dist(u_1,u_2)=3$.  
Since $|L(u_1)| + |L(u_2)| \geq 4 + 4 > 7=|L(v_1)|$, we can choose color $c_1$ for $u_1$ and color $c_2$ for $u_2$ such that
$|L(v_1)\setminus\{c_1,c_2\}|\geq 6$.  
Since $\ex(v_1)\geq 2$ and $\ex(v_2)\geq 1$, we can finish the coloring.
Hence, we can assume that $\dist(u_1,u_2) < 3$.  

\begin{figure}[!h!bt]
  \label{Lemma7-fig}
  \begin{center}
    \psfrag{V1}{{\normalsize$v_1$}}
    \psfrag{V2}{{\normalsize$v_2$}}
    \psfrag{V3}{{\normalsize$v_3$}}
    \psfrag{V4}{{\normalsize$v_4$}}
    \psfrag{U1}{{\normalsize$u_1$}}
    \psfrag{U2}{{\normalsize$u_2$}}
    \psfrag{U3}{{\normalsize$u_3$}}
    \psfrag{U4}{{\normalsize$u_4$}}
    \includegraphics{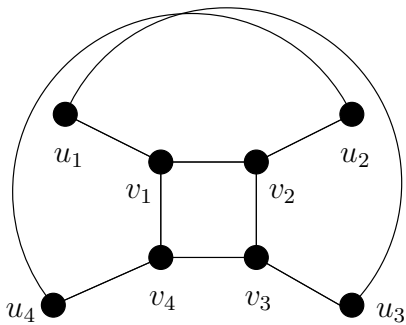}
  \end{center}
\caption{A 4-cycle with vertices $v_1$, $v_2$, $v_3$, $v_4$ and the adjacent
vertices not on the 4-cycle: $u_1$, $u_2$, $u_3$, $u_4$, respectively.  In Case 2 of Lemma~7, we also assume that vertices $u_1$ and $u_3$ are adjacent and that
 vertices $u_2$ and $u_4$ are adjacent.}
\end{figure}

Observe that $u_1$ and $u_2$ cannot be adjacent; otherwise $v_1$ would lie on two 4-cycles.
Thus, $u_1$ and $u_2$ must have a common neighbor.
By symmetry, we can assume that $u_1$ and $u_4$ have a common neighbor. 
Since $d(u_1)=3$ (and we have already accounted for two edges incident to $u_1$), vertices $u_1$, $u_2$, and $u_4$ must have a common neighbor $x$.  However, then $u_2$, $u_4$, and $x$ form a 3-cycle.  By Lemma~\ref{noC3}, this is a contradiction.
\end{proof}

\begin{lemma}
If $G$ is a non-Petersen 8-minimal subcubic graph, 
then $G$ does not contain two 5-cycles that share three consecutive vertices.
\label{no5cyclepath}
\end{lemma}
\begin{proof}
Suppose $G$ is a counterexample.
Taken together, the two given 5-cycles form a 6-cycle, with one
additional vertex adjacent to two vertices of the 6-cycle. Label the
vertices of the 6-cycle $v_1, v_2, \ldots, v_6$ and label the final
vertex $v_7$ (as in Figure~\ref{lemma8-fig}: Case 1).  Let $v_7$ be adjacent to $v_1$ and $v_4$.
We consider three cases, depending on how many pairs of vertices on the 6-cycle are distance 3 apart.
By Lemma~\ref{mainlemma}, we color all vertices of $G^2$ except the 7 $v_i$s.

\textit{Case 1:} Both $\dist(v_2, v_5)$ and $\dist(v_3, v_6)$ are at least 3.   Let $L(v)$ denote the list of remaining colors available at each uncolored vertex $v$.
In this case, $|L(v_1)|\geq 5$, $|L(v_4)|\geq 5$, $|L(v_7)| \geq5$ and $|L(v_2)|\geq 4$, $|L(v_3)|\geq 4$, $|L(v_5)|\geq 4$, $|L(v_6)|\geq 4$.  We assume equality holds.  We consider two subcases.

Subcase 1.1: $L(v_2) \cap L(v_5) \neq \emptyset$ or $L(v_3) \cap L(v_6) \neq \emptyset$.
Without loss of generality, we may assume that $L(v_2) \cap L(v_5) \neq \emptyset$.  Color $v_2$ and $v_5$ with some color $c_1\in L(v_2)\cap L(v_5)$. 
Since $|L(v_3)\setminus\{c_1\}|+|L(v_6)\setminus\{c_1\}|>|L(v_7)|$, we can 
choose color $c_2$ for $v_3$ and color $c_3$ for $v_6$ such that $|L(v_7)\setminus\{c_1,c_2,c_3\}|\geq 3$.  Greedily color the remaining vertices in the order $v_1$, $v_4$, $v_7$.

Subcase 1.2: $L(v_2) \cap L(v_5) = \emptyset$ and $L(v_3) \cap L(v_6) = \emptyset$. 
\begin{figure}[!h!bt]
  \label{lemma8-fig}
  \begin{center}
    \psfrag{V1}{{\normalsize$v_1$}}
    \psfrag{V2}{{\normalsize$v_2$}}
    \psfrag{V3}{{\normalsize$v_3$}}
    \psfrag{V4}{{\normalsize$v_4$}}
    \psfrag{V5}{{\normalsize$v_5$}}
    \psfrag{V6}{{\normalsize$v_6$}}
    \psfrag{V7}{{\normalsize$v_7$}}
    \psfrag{V8}{{\normalsize$v_8$}}
    \psfrag{V9}{{\normalsize$v_9$}}
    \psfrag{U2}{{\normalsize$u_2$}}
    \psfrag{U5}{{\normalsize$u_5$}}
    \psfrag{U7}{{\normalsize$u_7$}}
    \psfrag{U8}{{\normalsize$u_8$}}
    \psfrag{U9}{{\normalsize$u_9$}}
    \psfrag{X}{{\normalsize$x$}}
    \psfrag{Case1}{Case 1}
    \psfrag{Case2}{Case 2}
    \psfrag{Case3}{Case 3}
    \includegraphics{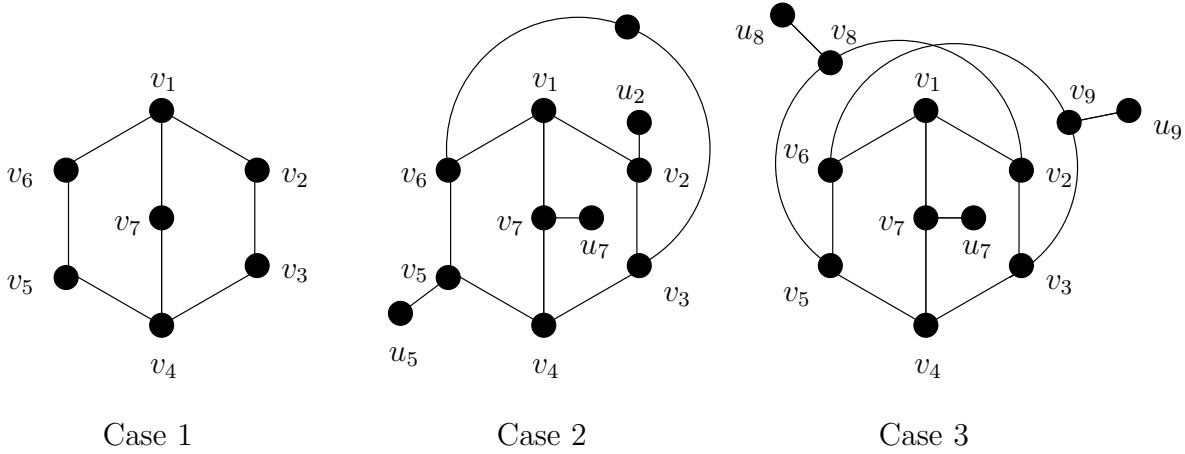}
  \end{center}
\caption{Lemma~8 considers two 5-cycles that share two consecutive edges.  
In Cases~2 and 3 of Lemma~8, we consider additional adjacencies.}
\end{figure}

Color $v_1, v_4, v_7$ so that no two vertices among $v_2, v_3, v_5, v_6$ have only one available color remaining.  Call these new lists $L'(v)$.
Note that $|L'(v_2)| +|L'(v_5)| \geq 5$ and $|L'(v_3)| +|L'(v_6)| \geq 5$. Hence we can color $v_2, v_3, v_5, v_6$.  

\textit{Case 2:} Exactly one of $\dist(v_2, v_5)$ or $\dist(v_3, v_6)$ is 2. 
Without loss of generality, we may assume that $\dist(v_2, v_5)\geq3$ and $\dist(v_3, v_6) = 2$.
Recall from Lemma~\ref{3reg} that $G$ is 3-regular.
Let $u_2$, $u_5$, and $u_7$ be the vertices not yet named that are adjacent to $v_2$, $v_5$, and $v_7$, respectively.
We cannot have $u_2=u_5$, since we have $\dist(v_2,v_5)\geq3$.
Note that $\dist(u_2,v_4)\geq3$ unless $u_2=u_7$.
Similarly, $\dist(u_5,v_1)\geq3$ unless $u_5=u_7$.
Moreover, we cannot have $u_2 = u_7$ or $u_5 = u_7$, since each equality implies that $G$ contains a 4-cycle.
Hence, $\dist(u_2,v_4)= 3$ and $\dist(u_5,v_1)= 3$ (see Figure~\ref{lemma8-fig}: Case 2).  Uncolor vertex $u_2$.
Let $L(v)$ denote the list of remaining available colors at each vertex $v$.
We have $|L(v_1)|\geq 6$, $|L(v_2)|\geq 5$, $|L(v_3)|\geq 6$, $|L(v_4)|\geq 5$,
$|L(v_5)|\geq 4$, $|L(v_6)|\geq 5$, $|L(v_7)|\geq 5$, and $|L(u_2)|\geq 2$.  We assume that equality holds.  We consider two subcases.

Subcase 2.1: $L(u_2) \cap L(v_4) \neq \emptyset$.
Color $u_2$ and $v_4$ with some color $c_1 \in L(u_2) \cap L(v_4)$.  Now choose color $c_2$ for $v_2$ and color $c_3$ for $v_5$ such that $|L(v_3)\setminus\{c_1,c_2,c_3\}|\geq 4$.  Let $L'(v)=L(v)\setminus\{c_1,c_2,c_3\}$.  The new lists satisfy $|L'(v_1)| \geq 3, |L'(v_3)| \geq 4, |L'(v_6)| \geq 2,
|L'(v_7)| \geq 2$.  Greedily color the remaining vertices in the order $v_7$, $v_6$, $v_1$, $v_3$.

Subcase 2.2: $L(u_2) \cap L(v_4) = \emptyset$.
We have two subcases here.  If $L(v_2) \cap L(v_5) \neq \emptyset$, then color $v_2$ and $v_5$
with a common color, and then color $u_2$ and $v_4$ to save a color at $v_3$. Now color the remaining vertices as in Subcase 2.1.
If $L(v_2) \cap L(v_5) = \emptyset$, then color $u_2$ and $v_4$ to save a color at $v_3$.
Now choose colors  for $v_6$ and for $v_7$ such that vertices $v_2$ and $v_5$ each have at least one remaining color.
Let $L'(v)$ denote the list of remaining available colors at each vertex $v$.
Note that $|L'(v_1)| \geq 2$, $|L'(v_3)| \geq 3$, and $|L'(v_2)| + |L'(v_5)| \geq 5$ since
$L(v_2) \cap L(v_5) = \emptyset$.  In each case, we can color $v_1, v_2, v_3, v_5$. 

\textit{Case 3:} Both $\dist(v_2, v_5)$ and $\dist(v_3, v_6)$ are 2.
Then $v_2$ and $v_5$ have a common neighbor, say $v_8$, and $v_3$ and $v_6$ have a common neighbor, say $v_9$ (see Figure~\ref{lemma8-fig}: Case 3).
Let $u_7$, $u_8$, and $u_9$ be the third vertices adjacent to $v_7$, $v_8$, and $v_9$, respectively.
We show that either $\dist(v_7,v_8)=3$ or $\dist(v_7,v_9)=3$ or $\dist(v_8,v_9)=3$.
Note that $\dist(v_7,v_8)=3$ unless $u_7=u_8$.
Similarly, $\dist(v_7,v_9)=3$ unless $u_7=u_9$ and $\dist(v_8,v_9)=3$ unless $u_8=u_9$.
However, we cannot have $u_7=u_8=u_9$, since $G$ is not the Petersen graph.  Hence, by symmetry, assume that $u_7\neq u_8$.
So $\dist(v_7,v_8)=3$.  
In this case we consider two other 5-cycles:
$v_1v_2v_3v_4v_7v_1$ and $v_2v_3v_4v_5v_8v_2$. 
These 5-cycles share three consecutive vertices; furthermore, because $d(v_7,v_8)=3$, we can finish either by Case 1 or Case 2 above.
\end{proof}

\begin{lemma}
\label{no5cycleedge}
If $G$ is a non-Petersen 8-minimal subcubic graph, then $G$ does not contain two 5-cycles that share an edge.
\end{lemma}
\begin{proof}
Suppose $G$ is a counterexample.
By Lemmas~\ref{3reg}-\ref{noC4}, we know that $G$ is 3-regular and that $g(G)\geq5$.  
Taken together, the two given 5-cycles form an 8-cycle, with a chord.
Label the vertices of the 8-cycle $v_1, v_2, \ldots, v_8$ with an edge between $v_1$ and $v_5$.
By Lemmas~\ref{noC4} and~\ref{no5cyclepath}, we know that $\dist(v_2,v_6)=3$.  Similarly, we know that $\dist(v_4,v_8)=3$.
By Lemma~\ref{mainlemma}, we color all vertices of $G^2$ except the 8 $v_i$s.
Let $L(v)$ denote the list of remaining available colors at each vertex $v$.
Note that $|L(v_1)|\geq 6$, $|L(v_2)|\geq 4$, $|L(v_3)|\geq 3$, $|L(v_4)|\geq 4$, $|L(v_5)|\geq 6$, $|L(v_6)|\geq 4$, $|L(v_7)|\geq 3$, and $|L(v_8)|\geq 4$.
We assume that equality holds.

\textit{Case 1:} There exists a color $c_1\in L(v_4)\cap L(v_8)$.  Use color $c_1$ on $v_4$ and $v_8$.
Since $|L(v_2)\setminus\{c_1\}|+|L(v_6)\setminus\{c_1\}|\geq 6$ and $|L(v_5)\setminus\{c_1\}|\geq5$, we can choose color $c_2$ for $v_2$ and color $c_3$ for $v_6$ such that $|L(v_5)\setminus\{c_1,c_2,c_3\}|\geq 4$.  Now since $\ex(v_1)\geq 1$ and $\ex(v_5)\geq 2$, we can finish the coloring by Lemma~\ref{extlemma}.
\begin{figure}[!h!bt]
  \label{Lemma9}
  \begin{center}
    \psfrag{V1}{{\normalsize$v_1$}}
    \psfrag{V2}{{\normalsize$v_2$}}
    \psfrag{V3}{{\normalsize$v_3$}}
    \psfrag{V4}{{\normalsize$v_4$}}
    \psfrag{V5}{{\normalsize$v_5$}}
    \psfrag{V6}{{\normalsize$v_6$}}
    \psfrag{V7}{{\normalsize$v_7$}}
    \psfrag{V8}{{\normalsize$v_8$}}
    \psfrag{U1}{{\normalsize$u_1$}}
    \psfrag{U2}{{\normalsize$u_2$}}
    \psfrag{U3}{{\normalsize$u_3$}}
    \psfrag{U4}{{\normalsize$u_4$}}
    \psfrag{U5}{{\normalsize$u_5$}}
    \psfrag{X}{{\normalsize$x$}}
    \psfrag{Lemma9}{\hspace{.15in}(a)}
    \psfrag{Lemma10}{\hspace{.20in}(b)}
    \includegraphics{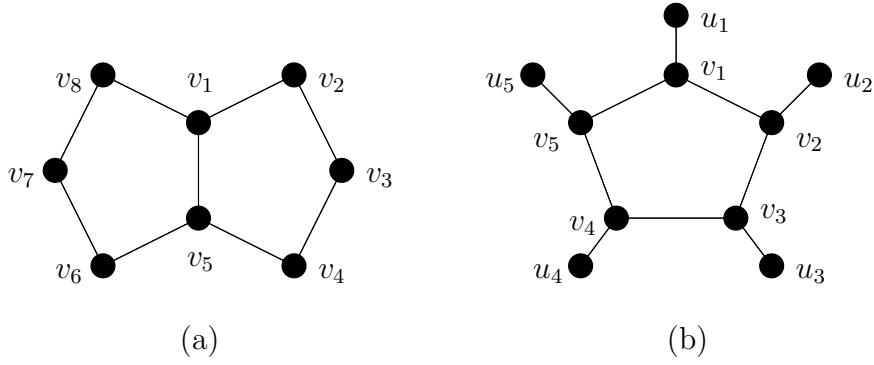}
  \end{center}
\caption{(a) Lemma~9 considers two 5-cycles that share an edges.  (b) Lemma~10 considers a single 5-cycle.}
\end{figure}

\textit{Case 2:} $L(v_4)\cap L(v_8)=\emptyset$.  We can choose color $c_1$ for $v_2$ and color $c_2$ for $v_6$ such that $|L(v_5)\setminus\{c_1,c_2\}|\geq 5$.  
Note that now $\ex(v_5)\geq 1$.
Now color $v_3$ and $v_7$ arbitrarily with colors from their lists; call them $c_3$ and $c_4$, respectively.
Since $L(v_4)\cap L(v_8)=\emptyset$, the remaining lists for $v_4$ and $v_8$ have sizes summing to at least 4; call these lists $L'(v_4)$ and $L'(v_8)$.  
If $|L'(v_4)|\geq 3$, then $\ex(v_4)\geq |L'(v_4)|-1=2$, so by Lemma~\ref{extlemma} we can finish the coloring.
Similarly, if $|L'(v_8)|\geq 3$, then $\ex(v_8)\geq |L'(v_8)|-1=2$, so by Lemma~\ref{extlemma} we can finish the coloring.
So assume that $|L'(v_4)|=|L'(v_8)|=2$.  Arbitrarily color $v_1$ from its list; call the color $c_3$.
Since $L(v_4)\cap L(v_8)=\emptyset$, either $|L'(v_4)\setminus\{c_3\}|=2$ or $|L'(v_8)\setminus\{c_3\}|=2$.
In the first case, $\ex(v_4)\geq 2$; in the second case, $\ex(v_8)\geq 2$.
In either case, we can greedily finish the coloring by Lemma~\ref{extlemma}.
\end{proof}

\begin{lemma}
If $G$ is a non-Petersen 8-minimal subcubic graph, then $g(G) > 5$.
\label{noC5}
\end{lemma}
\begin{proof}
Suppose $G$ is a counterexample.  
By Lemmas \ref{3reg}-\ref{noC4}, we know that $G$ is 3-regular and that $g(G)=5$.  
Let $v_1v_2v_3v_4v_5v_1$ be a 5-cycle and let $u_i$ be the neighbor of vertex $v_i$ not on the 5-cycle.

By Lemma~\ref{mainlemma}, we can greedily color all vertices except the $u_i$s and $v_i$s.  
Let $L(v)$ denote the list of remaining available colors at each vertex $v$.
Note that $|L(u_i)|\geq 2$ and $|L(v_i)|\geq 6$.  
We assume that equality holds for the $v_i$s.
By Lemma~\ref{no5cyclepath}, we know that $\dist(u_i,v_{i+2})=\dist(u_i,v_{i+3})=3$ for all $i$ (subscripts are modulo 5).  
By Lemma~\ref{no5cycleedge}
we also know that $\dist(u_i,u_{i+1})=3$.  

\textit{Case 1:} There exists a color $c_1\in L(u_1)\cap L(v_3)$.
Use $c_1$ on $u_1$ and $v_3$.
Greedily color vertices $u_2,u_3,u_4$; call these colors $c_2,c_3,c_4$, respectively.
Now $|L(v_1)\setminus\{c_1,c_2\}|=4$, $|L(v_2)\setminus\{c_1,c_2,c_3\}|\geq 3$, and $|L(u_5)|\geq2$.
We can choose color $c_5$ for $u_5$ and color $c_6$ for $v_2$ such that $|L(v_1)\setminus\{c_1,c_2,c_5,c_6\}|\geq 3$.  
Now greedily color the remaining vertices in the order $v_4,v_5,v_1$.

\textit{Case 2:} There exists a color $c_1 \in L(u_1)\cap L(u_2)$.  Use color $c_1$ on $u_1$ and $u_2$.  Now $|L(v_5)\setminus\{c_1\}| + |L(u_3)| > |L(v_2)\setminus\{c_1\}|$, so we can choose color $c_2$ for $v_5$ and color $c_3$ for $u_3$ so that excess($v_2)\geq 2$.  Note that excess($v_1)\geq 1$.  Hence, after we greedily color $u_5$, we can extend the partial coloring to the remaining uncolored vertices by Lemma~\ref{extlemma}.

\textit{Case 3:} $L(u_i)\cap L(u_{i+1})=\emptyset$ and $L(u_i)\cap L(v_{i+2})=\emptyset$ for all $i$.  By symmetry, we can assume $L(u_i)\cap L(v_{i+3})=\emptyset$ for all $i$.  
We now show that we can color each vertex with a distinct color.  Suppose not.

By Hall's Theorem~\cite{IGT}, there exists a subset of the uncolored vertices $V_1$ such that $|\cup_{v\in V_1}L(v)|<|V_1|$.
Recall that $|L(u_i)|\geq 2$ and $|L(v_i)|=6$ for all $i$.
Clearly, $|V_1| > 2$.  If $|V_1|\leq 6$, then $V_1\subseteq\{u_1,u_2,u_3,u_4,u_5\}$.
Any three $u_i$s contain a pair $u_j,u_{j+1}$; their lists are disjoint, so $|\cup_{v\in V_1}L(v)|\geq |L(u_j)|+|L(u_{j+1})|\geq 4$.  If $|V_1|=5$, then $V_1 =\{u_1,u_2,u_3,u_4,u_5\}$.  However, each color appears on at most two $u_i$s, hence $|\cup_{v\in V_1}L(v)|\geq 10/2 = 5$.
So say $|V_1|\geq 7$.  The Pigeonhole principle implies that $V_1$ must contain a pair $u_i, v_{i+2}$.  Since lists $L(u_i)$ and $L(v_{i+2})$ are disjoint, we have $|\cup_{v\in V_1}L(v)|\geq |L(u_i)| + |L(v_{i+2})| \geq 2 + 6 = 8$.  Hence, $|V_1| \geq 9$.  Now $V_1$ must contain a triple $u_i,u_{i+1},v_{i+3}$.  Since their lists are pairwise disjoint, we get $|\cup_{v\in V_1}L(v)|\geq |L(u_i)|+|L(u_{i+1})|+|L(v_{i+3})| \geq 2 + 2 + 6 = 10$.
This is a contradiction.  Thus, we can finish the coloring.
\end{proof}

Now we prove that if $G$ is 8-minimal, then $G$ does not contain a 6-cycle.  
\begin{lemma}
\label{noC6}
If $G$ is a non-Petersen 8-minimal subcubic graph, then $g(G) > 6$.
\end{lemma}
\begin{proof}
Let $G$ be a counterexample.
By Lemma~\ref{noC5}, we know that $g(G) > 5$.  
Hence, a counterexample must have girth 6.
We show how to color $G$ from lists of size 8.
First, we prove that if $H=C_6$, then $\chi_l(H^2)=3$.
Our plan is to first color all vertices except those on the 6-cycle, then color the vertices of the 6-cycle.

{\bf Claim:} If $H=C_6$, then $\chi_l(H^2)=3$.

Label the vertices $v_1$, $v_2$, $v_3$, $v_4$, $v_5$, $v_6$ in succession.
Let $L(v)$ denote the list of available colors at each vertex $v$.
We consider separately the cases where $L(v_1)\cap L(v_4) \neq \emptyset$ and where $L(v_1)\cap L(v_4)= \emptyset$.

\textit{Case 1:}
There exists a color $c_1\in L(v_1)\cap L(v_4)$.  Use color $c_1$ on $v_1$ and $v_4$.  
Note that $|L(v_i)\setminus\{c_1\}|\geq 2$ for each $i\in \{2,3,5,6\}$.
If there exists a color $c_2\in (L(v_2)\cap L(v_5)) \setminus\{c_1\}$, then use color $c_2$ on $v_2$ and $v_5$.  Now greedily color $v_3$ and $v_6$.
So suppose there is no color in $(L(v_2)\cap L(v_5)) \setminus\{c_1\}$.  Color $v_3$ arbitrarily; call it color $c_3$.
Either $|L(v_2)\setminus\{c_1,c_3\}|\geq 2$ or $|L(v_5)\setminus\{c_1,c_3\}|\geq 2$.
In the first case, greedily color $v_5, v_6, v_2$.
In the second case, greedily color $v_2, v_6, v_5$.

\textit{Case 2:}
$L(v_1)\cap L(v_4)=\emptyset$.  By symmetry, we assume ${L(v_2)\cap L(v_5)=\emptyset}$ and ${L(v_3)\cap L(v_6)=\emptyset}$.
Color $v_1$ arbitrarily; call it color $c_1$.  
If there exists $i$ such that $|L(v_i)\setminus\{c_1\}|=2$, then color $v_4$ from $c_2\in L(v_4)\setminus L(v_i)$; otherwise color $v_4$ arbitrarily.
Let $L'(v_j) = L(v_j) \setminus \{c_1,c_2\}$ for all $j\in \{2,3,5,6\}$.  Note that $|L'(v_2)| + |L'(v_5)|\geq 4$ and $|L'(v_3)| + |L'(v_6)|\geq 4$.
Also, note that there is at most one $k$ in $\{2,3,5,6\}$ such that $|L'(k)|=1$.  So by symmetry we consider two subcases.

Subcase 2.1: $|L'(v_j)|\geq 2$ for every $j\in \{2,3,5,6\}$.  We can finish as in case 1 above. 

Subcase 2.2: $|L'(v_2)| = 1$, $|L'(v_3)|\geq 2$, $|L'(v_6)|\geq 2$, and $|L'(v_5)|\geq 3$.
We color greedily in the order $v_2$, $v_3$, $v_6$, $v_5$.

This finishes the proof of the claim; now we prove the lemma.

Let $u$ and $v$ be adjacent vertices on a 6-cycle $\cyc$.  
By Lemma~\ref{mainlemma}, color all vertices except the vertices of $\cyc$.
Since $g(G)=6$, $\cyc$ has no chords.  Similarly, no two vertices of $\cyc$ have a common neighbor not on $\cyc$.  
Note that each vertex of $\cyc$ has at least three available colors.  Hence, by the Claim we can finish the coloring.
\end{proof}

The fact that $\chi_l(C_6^2)=3$ is a special case of a theorem by Juvan, Mohar, and \v{S}krekovski~\cite{juvan}.  
They showed that for any $k$, if $G=C_{6k}$, then $\chi_l(G^2) = 3$.
Their proof uses algebraic methods and is not constructive.
This fact is also a special case of a result by Fleischner and Stiebitz~\cite{fleischner}; their result also relies on algebraic methods.

\begin{lemma}
Let $\cyc$ be a shortest cycle in a non-Petersen 8-minimal subcubic graph $G$.  If $v$ and $w$ are each distance 1 from $\cyc$, then $v$ and $w$ are nonadjacent.
\end{lemma}
\begin{proof}
Let $\cyc$ be a shortest cycle in $G$.  Lemma~\ref{noC6} implies that $|V(\cyc)|\geq 7$.
Let $v_1,v_2,\ldots,v_k$ be the vertices of $\cyc$.  
Recall that $G$ is 3-regular.
Let $u_i$ be the neighbor of $v_i$ that is not on $\cyc$.  
Suppose that there exists $u_i$ adjacent to $u_j$.
Let $d$ be the distance from $v_i$ to $v_j$ along $\cyc$.  By combining the path $v_iu_iu_jv_j$ with the
shortest path along $\cyc$ from $v_i$ to $v_j$, we get a cycle of length $3+d\leq 3+\floor{|V(\cyc|)/2} < |V(\cyc)|$.  This contradicts the
fact that $\cyc$ is a shortest cycle in $G$.
\end{proof}

\begin{figure}[!h!bt]
  \begin{center}
    \psfrag{V1}{{\normalsize$v_{i-1}$}}
    \psfrag{V2}{{\normalsize$v_{i}$}}
    \psfrag{V3}{{\normalsize$v_{i+1}$}}
    \psfrag{V4}{{\normalsize$v_{i+2}$}}
    \psfrag{U1}{{\normalsize$u_{i-1}$}}
    \psfrag{U2}{{\normalsize$u_{i}$}}
    \psfrag{U3}{{\normalsize$u_{i+1}$}}
    \psfrag{U4}{{\normalsize$u_{i+2}$}}
    \includegraphics{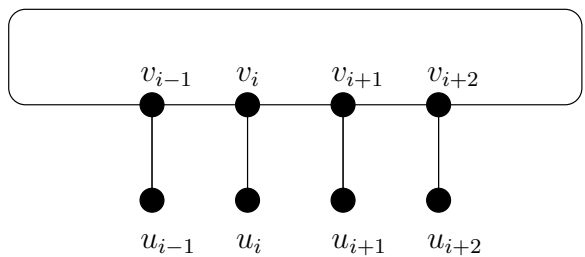}
  \end{center}
\caption{In the proof of Theorem~1, we frequently consider four consecutive
vertices on a cycle and their neighbors off the cycle.}
\end{figure}

We are now ready to prove Theorem~\ref{mainthm}. 

\begin{oneshot1}
If $G$ is a non-Petersen subcubic graph, then $\chi_l(G^2)\leq 8$.
\end{oneshot1}
\begin{proof}
Let $G$ be a counterexample.  By Lemma~\ref{3reg}, we know that $G$ is 3-regular.
By Lemma~\ref{noC6}, we know that $G$ has girth at least 7.
Let $\cyc$ be a shortest cycle in $G$.  
Let $v_1,v_2,\ldots, v_k$ be the vertices of $\cyc$.  Let $u_i$ be the neighbor of $v_i$ that is not on $\cyc$.  
Let $H$ be the union of the $v_i$s and the $u_i$s.
By Lemma~\ref{mainlemma}, we can color $G^2-H$.
Let $L(v)$ denote the list of available colors at each vertex $v$.
Note that $|L(v_i)|\geq 6$ and $|L(u_i)|\geq 2$ for all $i$.  We assume that equality holds.  In each of the following cases, we assume that none of the cases preceding it hold.

\textbf{Case 1}: If we can choose color $c_1$ for $u_i$ and color $c_2$ for $u_{i+1}$ such that $|L(v_i)\setminus\{c_1,c_2\}|\geq 5$ and $|L(v_{i+1})\setminus\{c_1,c_2\}|\geq 5$,
then we can extend the coloring to all of $G^2$.

Use colors $c_1$ and $c_2$ on $u_i$ and $u_{i+1}$.  Since $|L(u_{i-1})| = 2$ and $|L(v_{i+2})\setminus\{c_2\}|\geq 5$
and $|L(v_i)\setminus\{c_1,c_2\}|\geq5$, we can choose color $c_3$ for $u_{i-1}$ and color $c_4$ for $v_{i+2}$ so that $|L(v_i)\setminus\{c_1,c_2,c_3,c_4\}|\geq 4$.  
Color $u_{i+2}$ arbitrarily.
Now since $\excess(v_{i+1})\geq 1$ and $\excess(v_i)\geq 2$, we can greedily finish the coloring by Lemma~\ref{extlemma}.

\textbf{Case 2}: If we can choose color $c_1$ for $u_i$ such that $|L(v_i)\setminus\{c_1\}| = 6$, then we can extend the coloring to all of $G$.

Use color $c_1$ on $u_i$.  Since $|L(u_{i-1})|=2$ and $|L(v_{i+1})\setminus\{c_1\}|\geq 5$ and $|L(v_{i-1})\setminus\{c_1\}|\geq 5$, we can chose color $c_2$ for $u_{i-1}$ and
color $c_3$ for $v_{i+1}$ such that $|L(v_{i-1})\setminus\{c_1,c_2,c_3\}|\geq 4$.  If $c_2=c_3$, then we use $c_2$ on vertices $u_{i-1}$ and $v_{i+1}$; 
Now $\ex(v_{i-1})\geq 1$ and $\ex(v_i)\geq 2$.  So after we greedily color $u_{i+1}$, we can finish by Lemma~\ref{extlemma}.
Hence, we can assume $c_2\neq c_3$.
Note that either $c_2\not\in L(v_{i-1})$ or $c_3\not\in L(v_{i-1})$.
If $c_2\notin L(v_{i-1})$, then use $c_2$ on $u_{i-1}$; now we can finish by Case 1.  Hence, we can assume $c_3\notin L(v_{i-1})$.  Use $c_3$ on $v_{i+1}$.
Greedily color $u_{i+1}$ and $u_{i+2}$; call these colors $c_4$ and $c_5$, respectively.  We may assume that $|L(v_i)\setminus\{c_1,c_3,c_4\}| = 4$ (otherwise, we
can finish greedily as above).  We also know that $|L(u_{i-1})|=2$ and $|L(v_{i+2})\setminus\{c_3,c_4,c_5\}|\geq 3$.  Hence, we can choose color $c_6$ for $u_{i-1}$
and color $c_7$ for $v_{i+2}$ such that $|L(v_i)\setminus\{c_1,c_3,c_4,c_6,c_7\}|\geq 3$.  
Now since $\excess(v_{i-1})\geq 1$ and $\excess(v_i)\geq 2$, we can finish by Lemma~\ref{extlemma}.

\textbf{Case 3}: If we can choose color $c_1$ for $u_{i+1}$ such that $|L(v_i)\setminus\{c_1\}| = 6$, then we can extend the coloring to all of $G$.

Use color $c_1$ on $u_{i+1}$.  Since $|L(u_i)|=2$ and $|L(v_{i+2})\setminus\{c_1\}|\geq 5$ and $|L(v_{i+1}\setminus\{c_1\})|\geq5$, we can choose color $c_2$ for $u_i$ and color $c_3$ for
$v_{i+2}$ such that $|L(v_{i+1})\setminus\{c_1,c_2,c_3\}|\geq 4$.  Now we are in the same situation as in the proof of Case 2.  If $c_2=c_3$, then we use color $c_2$
on $u_i$ and $v_{i+2}$ and color greedily as in Case 2.  If $c_2\notin L(v_{i+1})\setminus\{c_1\}$, then we use $c_2$ on $u_i$ and we can finish by Case 1.
Hence we must have $c_3\notin L(v_{i+1})$.  Use $c_3$ on $L(v_{i+2})$.  As in Case 2, we have $|L(v_i)\setminus\{c_1,c_3\}|\geq5$ and $|L(v_{i+1})\setminus\{c_1,c_3\}|\geq5$.  Hence, we can finish as in Case 2.

\textbf{Remark}: Case 2 and Case 3 imply that for every $i$ we have $L(u_{i-1})\cup L(u_i)\cup L(u_{i+1})\subseteq L(v_i)$.  Furthermore, Case 1 shows
that $L(u_i)\cap L(u_{i+1}) = \emptyset$ for all $i$.  To show that $L(u_{i-1})$, $L(u_i)$, and $L(u_{i+1})$ are pairwise disjoint we prove Case 4.

\textbf{Case 4}: If we can choose color $c_1$ for $u_{i-1}$ and color $c_2$ for $u_{i+1}$ such that $|L(v_i)\setminus\{c_1,c_2\}|\geq 5$, then we can extend the coloring to $G$.

Use color $c_1$ on $u_{i-1}$ and color $c_2$ and $u_{i+1}$.  Since $|L(u_i)|=2$ and $|L(v_{i+2})\setminus\{c_2\}|\geq 5$ and $|L(v_{i+1})|=6$, we can choose color $c_3$ for
$u_i$ and color $c_4$ for $v_{i+2}$ such that $|L(v_{i+1})\setminus\{c_2,c_3,c_4\}|\geq 4$.  If $c_3=c_4$, then we use color $c_3$ on $u_i$ and $v_{i+2}$; since $\ex(v_{i+1})\geq 1$ and $\ex(v_i)\geq 2$, we can finish by Lemma~\ref{extlemma}.
So we assume $c_3\neq c_4$. Note that either $c_3\notin L(v_{i+1})$ or $c_4\notin L(v_{i+1})$.  

Suppose $c_3\notin L(v_{i+1})$.  Use $c_3$ on $u_i$.  
Since $|L(v_{i-1}) \setminus\{c_1,c_3\}|\geq 4$ and $|L(u_{i+2})|=2$ and $|L(v_{i+1})\setminus\{c_3\}| \geq 5$, we can choose color $c_5$ for $v_{i-1}$ and color $c_6$ for $u_{i+2}$
such that $|L(v_{i+1})\setminus\{c_2,c_3,c_5,c_6\}|\geq 4$.  
Now since $\ex(v_i)\geq 1$ and $\ex(v_{i+1})\geq 2$, we can finish by Lemma~\ref{extlemma}.

Suppose instead that $c_4\notin L(v_{i+1})$.
Use $c_4$ on $v_{i+2}$.  Color $u_{i+2}$ and $u_{i+3}$ arbitrarily; call these colors $c_5$ and $c_6$, respectively.  Since $|L(u_i)|=2$ and $|L(v_{i+3})\setminus\{c_4,c_5,c_6\}|\geq 3$
and $|L(v_{i+1})\setminus\{c_2,c_4,c_5\}|\geq4$, we can choose color $c_7$ for $u_i$ and color $c_8$ for $v_{i+3}$ such that $|L(v_{i+1})\setminus\{c_2,c_4,c_5,c_7,c_8\}|\geq 3$.
Now since $\ex(v_i)\geq 1$ and $\ex(v_{i+1})\geq 2$, we can finish by Lemma~\ref{extlemma}.

\textbf{Case 5}: We can extend the coloring to $G$ in the following way.  Color each $u_j$ arbitrarily; let $c(u_j)$ denote the color we use on each $u_j$.
Now assign a color to each $v_j$ from $L(u_j)\setminus\{c(u_j)\}$.

For each $j$, Case 4 implies that $L(u_{j-1})$, $L(u_j)$, and $L(u_{j+1})$ are pairwise disjoint.  Hence, each $v_j$ receives a color not in 
$\{c(u_{j-1}), c(u_j), c(u_{j+1})\}$. 
Similarly, since $L(u_j)$ is disjoint from $L(u_{j-2}), L(u_{j-1}), L(u_{j+1})$, and $L(u_{j+2})$, vertex $v_j$ receives a color not in
$\{c(v_{j-2}), c(v_{j-1}), c(v_{j+1}), c(v_{j+2})\}$.
Hence, the coloring of $G^2$ is valid.
\end{proof}

\section{Planar subcubic graphs with girth at least 7}
\label{planargirth7}

In this section we prove that if $G$ is a subcubic planar graph with girth at least 7, then $\chi_l(G^2)\leq 7$.
Lemma~\ref{girthlemma} implies that such a graph $G$ has $\Mad(G) < \frac{14}5$.
We exhibit four 7-reducible configurations.  We show that every subcubic graph with $\Mad(G)<\frac{14}5$ contains at least one of these 7-reducible configurations.
This implies the desired theorem.

\begin{lemma}
\label{7-reducible}
Let $G$ be a {\em minimal} subcubic graph such that $\chi_l(G^2) > 7$.  
For each vertex $v$, let $M_1(v)$ and $M_2(v)$ be the number of 2-vertices at distance 1 and distance 2 from $v$.
If $v$ is a 3-vertex, then $2M_1(v)+M_2(v)\leq 2$.  If $v$ is a 2-vertex, then $2M_1(v)+M_2(v)=0$.  
\end{lemma}
\begin{proof}
We list four 7-reducible configurations.  
We show that if there exists a vertex $v$ such that the quantity $2M_1(v)+M_2(v)$ is larger than claimed, then $G$ contains one
of the four 7-reducible configurations.

\textbf{Configuration 1}: If $G$ contains two adjacent 2-vertices $u_1$ and $u_2$, then $G[u_1u_2]$ is 7-reducible.  
Let $H=G-\{u_1,u_2\}$.  By hypothesis, $H^2$ has a proper coloring from its lists.  Now greedily color vertex $u_1$, then vertex $u_2$.

\textbf{Configuration 2}: If $G$ contains two 2-vertices $u_1$ and $u_2$ adjacent to a 3-vertex $v$, then $G[u_1u_2v]$ is 7-reducible.  
Let $H=G-\{u_1,u_2,v\}$.  
By hypothesis, $H^2$ has a proper coloring from its lists.  Now greedily color $v, u_1, u_2$ (in that order).

\textbf{Configuration 3}: If $G$ contains two adjacent 3-vertices $v_1$ and $v_2$ and each $v_i$ is adjacent to a distinct 2-vertex $u_i$, 
then $G[v_1v_2u_1u_2]$ is 7-reducible.  Let $H=G-\{v_1,v_2, u_1, u_2\}$.  By hypothesis, $H^2$ has a proper coloring from its lists.  
Now greedily color $v_1, v_2, u_1, u_2$.

\textbf{Configuration 4}: Suppose $G$ contains a 3-vertex $w$ that is adjacent to three 3-vertices $v_1$, $v_2$, and $v_3$.  
If each $v_i$ is adjacent to a distinct 2-vertex $u_i$, then $G[v_1v_2v_3u_1u_2u_3w]$ is 7-reducible.
Let $H=G-\{v_1,v_2, v_3, u_1, u_2, u_3, w\}$.  By hypothesis, $H^2$ has a proper coloring from its lists.  Now greedily color $v_1$, $v_2$, $v_3$, $w$, $u_1$, $u_2$, $u_3$.

If $v$ is a 2-vertex and $M_1(v) + M_2(v) > 0$, then $G$ contains
Configuration~1 or Configuration~2.  Hence $2M_1(v)+M_2(v)=0$ for every
2-vertex $v$.  If $v$ is a 3-vertex, then $M_1(v)>1$ yields Configuration~2.
If $M_1(v)=1$ and $M_2(v)\ge 1$, then $G$ contains Configuration 3.
If $M_1(v)=0$ and $M_2(v)\ge3$, then $G$ contains Configuration 4.
Hence $2M_1(v)+M_2(v)\leq 2$.
%
\end{proof}

\begin{theorem}
\label{2frac45}
If $G$ is a subcubic graph with $\Mad(G) < \frac{14}5$, then $\chi_l(G^2)\leq 7$.
\end{theorem}
\begin{proof}
Let $G$ be a minimal counterexample to the theorem.
By Lemma~\ref{7-reducible}, each 3-vertex $v$ satisfies $2M_1(v)+M_2(v)\leq 2$ and
each 2-vertex $v$ satisfies $2M_1(v)+M_2(v)=0$.
We show that these bounds imply $\Mad(G)\geq \frac{14}5$.
We use discharging to average out the vertex degrees, raising the degree
``assigned'' to each 2-vertex until every vertex is assigned at least $\frac{14}5$.
The initial charge $\mu(v)$ for each vertex $v$ is its degree.  We use a single
discharging rule:
\begin{enumerate}
\item[] \textbf{R1}: Each 3-vertex gives $\frac15$ to each 2-vertex at
distance 1 and gives $\frac1{10}$ to each 2-vertex at distance~2.
\end{enumerate}
Let $\mu^*(v)$ be the resulting charge at $v$.  Each 2-vertex has distance at
least 3 from every other 2-vertex.  If $d(v)=2$, we therefore have
${\mu^*(v) = 2 + 2(\frac15) + 4(\frac1{10}) = \frac{14}5}$.
Since $2M_1(v) + M_2(v)\leq 2$ when $d(v)=3$, we obtain
$\mu^*(v) = 3 - \frac15M_1(v) -\frac1{10}M_2(v) = 3 - \frac1{10}(2M_1(v) + M_2(v)) \geq 3-\frac15 = \frac{14}5$
in this case.
Since each vertex now has charge at least $\frac{14}5$, the average
degree is at least $\frac{14}5$, which is a contradiction.
\end{proof}

Theorem~\ref{2frac45} is best possible, since there exists a subcubic graph $G$ with $\Ad(G)$ equal to $\frac{14}5$ such that $G^2$ is not 7-colorable.  Form $G$ by removing a single edge from the Petersen graph.  Clearly, $\Ad(G)=\frac{14}5$; it is straightforward to verify that $\mad(G)=\frac{14}5$.  It is easy to show that $G^2$ contains a clique of size 8; hence, $G^2$ is not 7-colorable.

\begin{corollary}
If $G$ is a planar subcubic graph with girth at least 7, then $\chi_l(G^2)\leq 7$.
\end{corollary}
\begin{proof}
By Lemma~\ref{girthlemma}, $\Mad(G) < \frac{14}5$.  By Theorem~\ref{2frac45}, this implies that $\chi_l(G^2)\leq 7$.
\end{proof}

\section{Planar subcubic graphs with girth at least 9}
\label{planargirth9}

In this section we prove that if $G$ is a subcubic planar graph with girth at least 9, then $\chi_l(G^2)\leq 6$.
Lemma~\ref{girthlemma} implies that such a graph $G$ has $\Mad(G) < \frac{18}7$.
In fact we prove the following stronger result: If $G$ is a subcubic graph with girth at least 7 and $\Mad(G)<\frac{18}7$, then $\chi_l(G^2)\leq 6$.
The restriction of girth at least 7 is necessary to ensure for example that vertices $u_1$ and $u_4$ in Figure 6b are distance at least 3 apart.
Recall that a configuration is $6'$-reducible if it cannot appear in a 6-minimal graph with girth at least 7.
We exhibit four $6'$-reducible configurations.  
We show that every subcubic graph with $\Mad(G)<\frac{18}7$
contains at least one of these $6'$-reducible configurations.
This implies the desired theorem.

We will prove that the following four configurations (shown in Figures 6a, 6b, 7a, and 7b) are $6'$-reducible.
We begin with a definition:
If $v$ is a 3-vertex, then we say that $v$ is of \textit{class $i$} if $v$ is adjacent to $i$ vertices of degree 2.  
Note that if $v_1$ and $v_2$ are adjacent 2-vertices, then $G[v_1v_2]$ is $6'$-reducible.  
Hence, we assume that no pair of 2-vertices is adjacent.

\begin{figure}[!h!bt]
  \label{two6reduc}
  \begin{center}
 \psfrag{V1}{{\normalsize$u_1$}}
    \psfrag{V2}{{\normalsize$u_2$}}
    \psfrag{V3}{{\normalsize$v_1$}}
    \psfrag{V4}{{\normalsize$v_2$}}
    \psfrag{V5}{{\normalsize$u_3$}}
    \psfrag{V6}{{\normalsize$u_4$}}
    \psfrag{W1}{{\normalsize$u_1$}}
    \psfrag{W2}{{\normalsize$u_2$}}
    \psfrag{H1}{{\normalsize$u_3$}}
    \psfrag{W3}{{\normalsize$v_1$}}
    \psfrag{W4}{{\normalsize$v_2$}}
    \psfrag{W5}{{\normalsize$u_4$}}
    \psfrag{Left}{\hspace{.27in}(a)}
    \psfrag{Right}{\hspace{.40in}(b)}
    \includegraphics{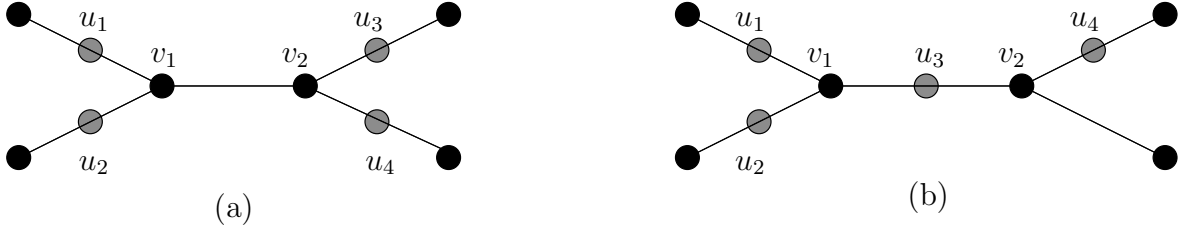}
  \end{center}
\caption{Two $6'$-reducible subgraphs. (a) Two adjacent class 2 vertices $v_1$ and $v_2$.
(b) A class 3 vertex $v_1$ and a class 2 vertex $v_2$ at distance 2.}
\end{figure}
\begin{lemma}

\label{6reduc-1}
If $v_1$ and $v_2$ are adjacent class 2 vertices, then $G[v_1v_2]$ is $6'$-reducible.
(This configuration is shown on the left in Figure~6.)
\end{lemma}
\begin{proof}
Let $v_1$ and $v_2$ be adjacent class 2 vertices.
Let $v_1$ be adjacent to vertices $u_1$ and $u_2$ and let $v_2$ be adjacent to vertices $u_3$ and $u_4$.
Let $H=G-\{v_1,v_2,u_1,u_2,u_3,u_4\}$.
By hypothesis, $H^2$ has a coloring from its lists.
Let $L(x)$ denote the list of remaining available colors for each uncolored vertex $x$ in $G$.
Note that $|L(v_i)|\geq 4$ and $|L(u_i)|\geq 3$ for each $i$.  We assume that equality holds (otherwise we discard colors until it does).
Since $G$ has girth at least 7, note that $u_1$ and $u_2$ are each distance 3 from each of $u_3$ and $u_4$. 

Since $|L(v_1)| = 4$ and $|L(u_1)| = 3$, there is a color $c\in L(v_1) \setminus L(u_1)$.
Use color $c$ on vertex $v_1$.  The sizes of the new lists (after removing $c$ from each) are $|L(u_1)\setminus\{c\}| = 3$, 
$|L(v_2)\setminus\{c\}| \geq 3$, and $|L(u_i)\setminus\{c\}| \geq 2$ for $i=2,3,4$. 
Greedily color the remaining vertices in the order $u_3, u_4, v_2, u_2, u_1$.
\end{proof}

\begin{lemma}
\label{6reduc-2}
If $v_1$ is a class 3 vertex, $v_2$ is either a class 2 or class 3 vertex, and vertices $v_1$ and $v_2$ have a common
neighbor $u_3$, then $G[v_1v_2u_3]$ is $6'$-reducible.
(This configuration is shown on the right in Figure 6.)
Moreover, if $G$ contains this configuration and $G^2-u_3$ has a proper $L$-coloring from lists $L$ of size 6, then $G^2$ has two proper
$L$-colorings $\phi$ and $\psi$ such that $\phi(u_3)\neq \varphi(u_3)$.
\end{lemma}
\begin{proof}
Let $v_1$ be a 3-vertex adjacent to three 2-vertices $u_1$, $u_2$, and $u_3$.
Suppose that $v_2$ is a 3-vertex adjacent to $u_3$ and also adjacent to another 2-vertex, $u_4$.
Let $H=G-\{v_1,v_2,u_1,u_2,u_3,u_4\}$.
By hypothesis, $H^2$ has a coloring from its lists.
Let $L'(x)$ denote the list of remaining available colors for each uncolored vertex $x$ in $G$.
Note that $|L'(u_1)|\geq 3$, $|L'(u_2)|\geq 3$, $|L'(u_3)|\geq 5$, $|L'(u_4)|\geq 2$, $|L'(v_1)|\geq 4$, and $|L'(v_2)|\geq 2$.
We assume that equality holds.
(Since $G$ has girth at least 7, note that $u_4$ is distance at least 3 from each of $u_1$, $u_2$, and $v_1$.)

Since $|L'(v_1)|=4$ and $|L'(u_1)|=3$, we can choose a color $c\in L'(v_1)\setminus L'(u_1)$.  Use color $c$ on vertex $v_1$.  Greedily color vertex $v_2$, then vertex $u_4$.  At this point, vertex $u_3$ has at least two available colors.  We can use either available color on $u_3$ (one choice will give coloring $\phi$ and the other will give coloring $\psi$).  Now greedily color vertex $u_2$, then vertex $u_1$. 
\end{proof}

\begin{figure}[!h!bt]
  \begin{center}
    \psfrag{A1}{{\normalsize$v_1$}}
    \psfrag{A2}{\vspace{.2in}{\normalsize$v_2$}}
    \psfrag{A3}{{\normalsize$v_3$}}
    \psfrag{B1}{{\normalsize$u_1$}}
    \psfrag{B2}{{\normalsize$u_5$}}
    \psfrag{B3}{{\normalsize$u_2$}}
    \psfrag{B4}{{\normalsize$u_3$}}
    \psfrag{B5}{{\normalsize$u_4$}}
    \psfrag{U}{{\normalsize $v_1$}}
    \psfrag{V1}{{\normalsize$v_2$}}
    \psfrag{V2}{}
    \psfrag{V3}{}
    \psfrag{S1}{{\normalsize$u_1$}}
    \psfrag{S2}{{\normalsize$u_2$}}
    \psfrag{S3}{{\normalsize$u_3$}}
    \psfrag{W1}{{\normalsize$u_4$}}
    \psfrag{W2}{{\normalsize$u_5$}}
    \psfrag{W3}{{\normalsize$v_3$}}
    \psfrag{W4}{{\normalsize$v_4$}}
    \psfrag{W5}{{\normalsize$u_6$}}
    \psfrag{Left}{(a)}
    \psfrag{Right}{(b)}
    \psfrag{X}{}
    \psfrag{Y}{}
    \psfrag{Z}{}
    \includegraphics{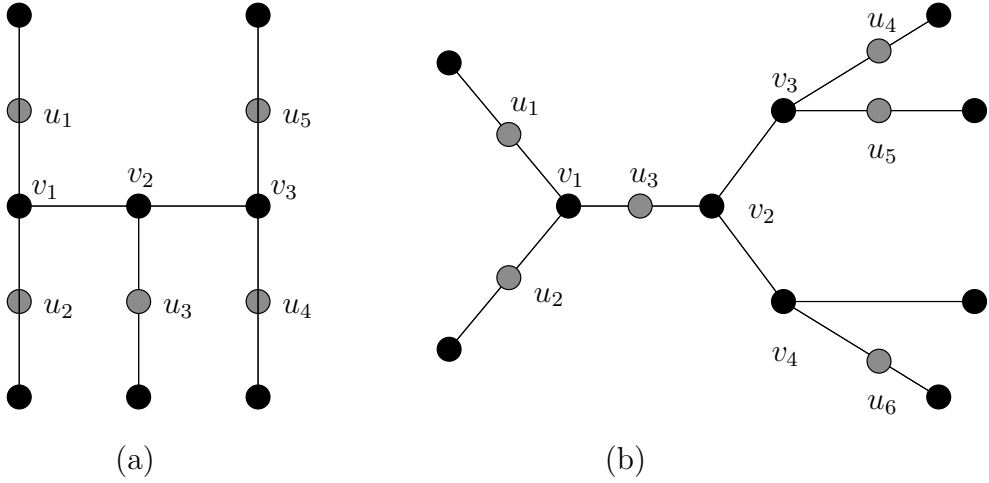}
  \end{center}
\caption{An $H$-configuration and a $Y$-configuration; both configurations are $6'$-reducible.
(a) An $H$-configuration: a class 1 vertex $v_2$ is adjacent to two class 2 vertices $v_1$ and $v_3$.
(b) A Y-configuration: a class 1 vertex $v_2$ is adjacent to a class 2 vertex $v_3$ and a class 1 vertex $v_4$, and is distance two from a class 3 vertex $v_1$.} 
\label{Y-conf}
\label{H-conf}
\end{figure}

\begin{lemma}
\label{6reduc-3}
We use the term $H$-configuration to denote a class 1 vertex adjacent to two class 2 vertices.
An $H$-configuration (shown on the left in Figure~\ref{H-conf}) is $6'$-reducible.
\end{lemma}
\begin{proof}
Let $H=G-\{v_1,v_2,v_3,u_1,u_2,u_3,u_4,u_5\}$ (see Figure 7).
By hypothesis, $H^2$ has a coloring from its lists.
Let $L(x)$ denote the list of remaining available colors for each uncolored vertex $x$ in $G$.
Note that $|L(u_i)|\geq 3$, $|L(v_1)|\geq 4$, $|L(v_3)|\geq 4$, and $|L(v_2)|\geq 5$.  We assume that equality holds. 
Since $|L(v_2)|> |L(u_5)|$, we can choose color $c\in L(v_2)\setminus L(u_5)$.
Use color $c$ on vertex $v_2$.  Now greedily color the remaining vertices in the order $u_1,u_2,v_1,u_3,v_3,u_4,u_5$.
\end{proof}

\begin{lemma}
\label{6reduc-4}
We use the term $Y$-configuration to denote a class 1 vertex adjacent to a class 2 vertex, adjacent to a class 1 vertex, and distance two from a class 3 vertex.
A $Y$-configuration (shown on the right in Figure~\ref{Y-conf}) is $6'$-reducible.
\end{lemma}
\begin{proof}
Let $H=G-\{v_1,u_1,u_2,u_3\}$ (see Figure 7).
By hypothesis, $H^2$ has a proper coloring from its lists.
Let $L(x)$ denote the list of remaining available colors for each uncolored vertex $x$.
Assume the coloring of $H^2$ cannot be extended to $G^2$. 
Hence $|L(v_1)| = |L(u_1)| = |L(u_2)| = |L(u_3)|= 3$ and $L(v_1) = L(u_1) = L(u_2) = L(u_3)$.  (Otherwise the coloring could be extended to $G^2$.)
By Lemma~\ref{6reduc-2}, $H^2$ has a recoloring  such that $v_2$ gets a
different color than it currently has.
Under this recoloring of $H^2$, the lists of available colors for $u_1$ and $v_1$ are no longer identical.
Hence, the recoloring of $H^2$ can be extended to $G^2$.   
\end{proof}

\begin{theorem}
\label{2frac47}
If $G$ is a subcubic graph with $\Mad(G) < \frac{18}7$ and girth at least 7, then $\chi_l(G^2)\leq 6$.
\end{theorem}
\begin{proof}
Let $G$ be a minimal counterexample to Theorem~\ref{2frac47}.
We show that if $G$ does not contain any of the $6'$-reducible configurations in Lemmas~\ref{6reduc-1},~\ref{6reduc-2},~\ref{6reduc-3},~and~\ref{6reduc-4},
then $\Mad(G) \geq \frac{18}7$.  We use a discharging argument with initial charge $\mu(v) = d(v)$.  
We have the following three discharging rules.
\begin{enumerate}
\item[] \textbf{R1}: Each 3-vertex gives $\frac27$ to each adjacent 2-vertex. 
\item[] \textbf{R2}: Each class 0 vertex gives $\frac17$ to each adjacent 3-vertex. 
\item[] \textbf{R3}: Each class 1 vertex gives $\frac17$ to each adjacent class 2 vertex and gives $\frac17$ to each class 3 vertex at distance~2.
\end{enumerate}

We must show that for every vertex $v$, the new charge $\mu^*(v)\geq \frac{18}7$.

Recall that each 2-vertex $v$ is adjacent only to 3-vertices. 
Hence, for a 2-vertex $v$ we have $\mu^*(v) = 2 + 2(\frac27) = \frac{18}7$.
So we consider 3-vertices.

Let $v$ be a 3-vertex.  We consider vertices of class 0, class 1, class 2, and class 3 separately.

If $v$ is class 0, then $\mu^*(v) = 3 - 3(\frac17) = \frac{18}7$.

If $v$ is class 2, then by Lemma~\ref{6reduc-1} vertex $v$ is adjacent to a class 1 vertex or a class 0 vertex.  
Hence $\mu^*(v) = 3 - 2(\frac27) + \frac17 = \frac{18}7$.

If $v$ is class 3, then by Lemma~\ref{6reduc-2} each 3-vertex at distance 2 from $v$ is a class 1 vertex.
Hence $\mu^*(v) = 3 - 3(\frac27) + 3\frac17 = \frac{18}7$.

Let $v$ be class 1.  By Lemma~\ref{6reduc-3}, $v$ is adjacent to at most one class 2 vertex.  
Clearly, $v$ is distance 2 from at most one class 3 vertex.
Hence $\mu^*(v)\geq\frac{18}7$ unless $v$ is adjacent to a class 2 vertex $w$ and distance 2 from a class 3 vertex $x$.  So we consider this case.
Let $y$ be the other 3-vertex adjacent to $v$.  Clearly, $y$ is not class 3 or class 2 (by Lemma~\ref{6reduc-3}).  If $y$ is class 1, then we have the $6'$-reducible subgraph
in Lemma~\ref{6reduc-4}.  Hence, $y$ must be class 0.  In that case $y$ gives $\frac17$ to $v$, so $\mu^*(v) = 3 - \frac27 - 2(\frac17) + \frac17 = \frac{18}7$.
Thus, $\Mad(G)\geq \frac{18}7$.
This is a contradiction, so no counterexample exists.  
\end{proof}

\begin{corollary}
If $G$ is a planar subcubic graph with girth at least 9, then $\chi_l(G^2)\leq 6$.
\end{corollary}
\begin{proof}
From Lemma~\ref{girthlemma}, we see that $\Mad(G) < \frac{18}7$.  By Theorem~\ref{2frac47}, this implies that $\chi_l(G^2)\leq 6$.
\end{proof}

\section{Efficient Algorithms}

Since the proof of Theorem~\ref{mainthm} colors all but a constant number of vertices greedily, it is not surprising that the algorithm can be made to run in linear time.  For completeness, we give the details.

If $G$ is not 3-regular or $G$ has girth at most 6, then we find a small subgraph $H$ (one listed in Lemmas~\ref{3reg}-\ref{noC6}) that contains a low degree vertex
 or a shortest cycle.  It is easy to greedily color $G^2- V(H)$ in linear time (for example, using breadth-first search).  Since $H$ has constant size, we can finish the coloring in constant time.

Suppose instead that $G$ is 3-regular and has girth at least 7.
Choose an arbitrary vertex $v$.  Find a shortest cycle through $v$ (for example, using breadth-first search); call it $\cyc$.  Let $H$ consist of $\cyc$ and vertices at distance 1 from $\cyc$.  We greedily color $G^2- V(H)$ in linear time.
Using the details given in the proof of Theorem~\ref{mainthm}, we can finish the coloring in time linear in the size of $H$.

The proofs of Theorems~\ref{2frac45} and~\ref{2frac47} are examples of a large class of discharging proofs
that can be easily translated into linear time algorithms.
The algorithm for each consists of finding a reducible configuration $H$
(7-reducible for Theorem~\ref{2frac45} and $6'$-reducible for Theorem~\ref{2frac47}), recursively coloring $G^2- V(H)$, then extending the coloring to $G^2$.  To achieve a linear running time, we need to find the reducible configuration in amortized constant time.
We make no effort to discover the optimal constant $k$ in the $kn$ running time;
we only outline the technique to show that the algorithm can be made to run in
linear time.

First we decompose $G$, by removing one reducible configuration after another;
when we remove a configuration from $G$, we add it to a list $A$ (of removed configurations).
After decomposing $G$, we build the graph back up, adding elements of $A$ in the reverse of the order they were removed.  When we add back an element of $A$, we color all of its vertices.  In this way, we eventually reach $G$, with every vertex colored.  
We call these two stages the decomposing phase and the rebuilding phase.
It only remains to specify how we find each configuration that we remove during the decomposing phase.

Our plan is to maintain a list $B$ of instances in the graph of reducible 
configurations.  We begin with a preprocessing phase, in which we store in $B$ 
 every instance of a reducible configuration in the original graph.  Using brute force, we can do this in linear time (since we have only a constant number of reducible configurations and each configuration is of bounded size, each vertex can appear in only a constant number of instances of reducible configurations).

When we remove a reducible configuration $H$ from $G$, we may create new
reducible configurations.  We can search for these new reducible configurations 
in constant time (since they must be adjacent to $H$).  
We add each of these new reducible configurations to $B$.
In removing $H$, we may have destroyed one or more reducible configurations in $B$ (for example, if they contained vertices of $H$).
We make no effort to remove the destroyed configurations from $B$.
Thus, at every point in time, $B$ will contain all the reducible configurations in the remaining graph (as well as possibly containing many ``destroyed'' reducible configurations).
To account for this, when we choose a configuration $H$ from $B$ to remove from 
the remaining graph, we must verify that $H$ is not destroyed.
If $H$ is destroyed, we discard it, and proceed to the next configuration in $B$.
We will show that the entire process of decomposing $G$ (and building $A$) takes
linear time.  (However, during the process, the time required to find a particular configuration to add to $A$ may not be constant.)

Theorems~\ref{2frac45} and~\ref{2frac47} guarantee that as we decompose $G$, list $B$ will never be empty.
Our only concern is that perhaps $B$ may contain ``too many'' destroyed configurations.  We show that througout both the preprocessing phase and the decomposing phase, only a linear number of configurations get added to $B$.
In the original graph $G$, each vertex can appear in only a constant number of reducible configurations; hence, in the preprocessing phase, only a linear number of reducible configurations are added to $B$.

During the decomposing phase, if we remove a destroyed configuration from $B$,
we discard it without adding any configurations to $B$.  If we remove a valid
configuration from $B$, we add only a constant number of configurations to $B$.
Each time we remove a valid configuration from $B$, we decrease the number of
vertices in the remaining graph; hence we remove only a linear number of valid
configurations from $B$.  Thus, during the decomposing phase, we add only a 
linear number of configurations to $B$.  As a result, the decomposing phase
runs in linear time.

During the rebuilding phase, we use constant time to add a configuration back,
and constant time to color the configuration's vertices (we do this using the lemma that proved the configuration was reducible).  List $A$ contains only a linear number of configurations, hence, the rebuilding phase runs in linear time.

\section{Future Work}

As we mentioned in the introduction, Theorem~\ref{mainthm} is best possible, since there are infinitely many connected subcubic graphs $G$ such that $\chi_l(G^2)=8$ (for example, any graph which contains the Petersen graph with one edge removed).  However, it is natural to ask whether the result can be extended to graphs with arbitrary maximum degree.  Let $G$ be a graph with maximum degree $\Delta(G)=k$.  Since $\Delta(G^2)\leq k^2$, we immediately get that $\chi_l(G^2)\leq k^2+1$.  If $G^2\neq K_{k^2+1}$, then by the list-coloring version of Brooks' Theorem~\cite{ERT}, we have $\chi_l(G^2)\leq k^2$.  Hoffman and Singleton~\cite{hoffman} made a thorough study of graphs $G$ with maximum degree $k$ such that $G^2=K_{k^2+1}$.  They called these \textit{Moore Graphs}.  They showed that a unique Moore Graph exists when $\Delta(G)\in\{2,3,7\}$ and possibly when $\Delta(G)=57$ (which is unknown), but that no Moore Graphs exist for any other value of $\Delta(G)$.  (When $\Delta(G)=3$, the unique Moore Graph is the Petersen Graph).  Hence, if $\Delta(G)\not\in\{2,3,7,57\}$, we know that $\chi_l(G^2)\leq\Delta(G)^2$.  As in Theorem~\ref{mainthm}, we believe that we can improve this upper bound.

\begin{conjecture}
If $G$ is a connected graph with maximum degree $k\geq3$ and $G$ is not a Moore Graph, then $\chi_l(G^2)\leq k^2-1$.
\label{conj}
\end{conjecture}

Erd\H{o}s, Fajtlowitcz and Hoffman~\cite{EFH} considered graphs $G$ with maximum degree $k$ such that $G^2=K_{k^2}$.  The proved the following result, which
provides evidence in support of our conjecture.

\begin{oneshot}
(Erd\H{o}s, Fajtlowitcz and Hoffman~\cite{EFH})
Apart from the cycle $C_4$, there is no graph $G$ with maximum degree $k$ such
that $G^2=K_{k^2}$.
\end{oneshot}

We extend this result to give a bound on the clique number $\omega(G^2)$ of 
the square of a non-Moore graph $G$ with maximum degree $k$.

\begin{lemma}
If $G$ is a connected graph with maximum degree $k\geq 3$ and $G$ is not a Moore graph, then the clique number $\omega(G^2)$ of $G^2$ is at most $k^2-1$.
\end{lemma}
\begin{proof}
If $G$ is a counterexample, then by the Theorem of Erd\H{o}s, Fajtlowitcz and Hoffman, we know that $G^2$ properly contains a copy of $K_{k^2}$.
Choose adjacent vertices $u$ and $v_1$ such that $v_1$ is in a clique of size $k^2$ (in $G^2$) and $u$ is not in that clique; call the clique $H$.
Note that $|N[v_1]|\leq k^2+1$, so all vertices in $N[v_1]$ other than $u$ must be in $H$.
Label the neighbors of $u$ as $v_i$s.  Note that no $v_i$ is on a 4-cycle.  If so, then $|N[v_i]|\leq k^2$; since $u\in N[v_i]$ and $u\not\in V(H)$, we get $|V(H)|\leq k^2-1$, which is a contradiction.

Note that each neighbor of a vertex $v_i$ (other than $u$) must be in $H$.  Since no $v_i$ lies on a 4-cycle, each pair $v_i,v_j$ have $u$ as their only common neighbor.  So the $v_i$s and their neighbors (other than $u$) are $k^2$ vertices in $H$.  But $u$ is within distance 2 of each of these $k^2$ vertices in $H$.  Hence, adding $u$ to $H$ yields a clique of size $k^2+1$.  This is a contradiction.
\end{proof}

We believe that Conjecture~\ref{conj} can probably be proved using an argument similar to our proof of Theorem~\ref{mainthm}.  In fact, arguments from our proof of Theorem~\ref{mainthm} easily imply that if $G$ is a counterexample to Conjecture~\ref{conj}, then $G$ is $k$-regular and has girth either 4 or 5. 
However, we do not see a way to handle these remaining cases without resorting to extensive case analysis (which we have not done).

Significant work has also been done proving lower bounds on $\chi_l(G)$.
Brown~\cite{brown} constructed a graph $G$ with maximum degree $k$ and $\chi_l(G^2)\geq k^2-k+1$ whenever $k-1$ is a prime power.  By combining results of Brown~\cite{brown} and Huxley~\cite{huxley}, Miller and \v{S}ir\'{a}\v{n}~\cite{miller} showed
that for every $\epsilon>0$ there exists a constant $c_{\epsilon}$ such that for every $k$ there exists a graph $G$ with maximum degree $k$ such that $\chi_l(G^2)\geq k^2 - c_{\epsilon}k^{19/12+\epsilon}$.

Another area for further work is reducing the girth bounds imposed in 
Theorems~\ref{2frac45} and~\ref{2frac47}.  
We know of no subcubic planar graph $G$ with girth at least 4 such that $\chi_l(G^2)=7$.  (If $G$ is the cartesian product $C_3\Box K_2$, then subdividing an edge of $G$ not in a 3-cycle yields a planar subcubic graph $G'$ such that $\chi_l((G')^2)=7$).
We know of no subcubic planar graph $G$ with girth at least 6 such that $\chi_l(G^2)= 6$.

Finally, we can consider the restriction of Theorem~\ref{mainthm} to planar graphs.
If $G$ is a planar subcubic graph, then we know that $\chi_l(G^2)\leq 8$.
However, we do not know of any planar graphs for which this is tight.
This returns us to the question that motivated much of this research
and that remains open.
\begin{question}
Is it true that every planar subcubic graph $G$ satisfies $\chi_l(G^2)\leq 7$?
\end{question}

It is easy to show that Question~2 is equivalent to the analagous question for planar cubic graphs.  To prove this, we show how to extend a planar subcubic graph to a planar cubic graph.  Let $G$ be a planar subcubic graph with a vertex $v$ of degree at most 2.  Let $J$ be the graph formed by subdividing an edge of $K_4$ and let $u$ be the 2-vertex in $J$.  If $d(v)=1$, associate vertices $u$ and $v$; if $d(v)=2$, instead add an edge between $u$ and $v$.  By repeating this process for each 1-vertex and 2-vertex in $G$, we reach a planar cubic graph.

\section{Acknowledgements}
We would like to thank A.V. Kostochka, D.B. West, and an anonymous referee for their invaluable comments.  After proving Theorem~\ref{2frac47}, we learned that Fr\'{e}d\'{e}ric Havet~\cite{havet} has proved the same result.



\begin{thebibliography}{99}

\bibitem{borodin} O.V. Borodin, A.N. Glebov, A.O. Ivanova, T.K. Neustroeva, V.A. Tashkinov,
Sufficient conditions for planar graphs to be 2-distance
$(\Delta+1)$-colorable. (Russian) {\em Sib. Elektron. Mat. Izv.}
{\bf 1} (2004), 129--141

\bibitem{brown} W.G. Brown, On graphs that do not contain a Thomsen graph,
{\em Canad. Math. Bull} {\bf 9} (1966) 281--285.

\bibitem{DST} Z. Dvo\v{r}\'{a}k, R. \v{S}krekovski, M. Tancer, List-Coloring
squares of sparse subcubic graphs, submitted for publication.

\bibitem{kral}Z. Dvo\v{r}\'{a}k, D. Kr\'{a}l, P. Nejedl\'{y}, and R. \v{S}krekovski,
Coloring squares of planar graphs with no short cycles, submitted
for publication.

\bibitem{ERT} P. Erd\H{o}s, A.L. Rubin, and H. Taylor,
Choosability in graphs, {\em Proceedings of the West Coast Conference on Combinatorics, Graph Theory and Computing}, Arcata, California, 1979, Congressus Numeratium, {\bf 26} (1980), 125--157.

\bibitem{EFH} P. Erd\H{o}s, S. Fajtlowicz and A.J. Hoffman, Maximum degree in graphs of diameter 2, {\em Networks} {\bf 10} (1980) 87--90.

\bibitem{fleischner}
H. Fleischner and M. Stiebitz, A solution to a colouring problem of P. Erd\H{o}s.  Special volume to mark the centennial of Julius Petersen's ``Die Theorie der regul\"{a}ren Graphs'', Part II. {\em Discrete Math.} {\bf 101} (1992), no. 1-3, 39--48.

\bibitem{havet}
F. Havet, Choosability of the square of planar subcubic graphs with large girth,
{\em Technical Report}, INRIA, 2006.

\bibitem{havet2}
F. Havet, J. van den Heuvel, C. McDiarmid, and B. Reed, List colouring squares of planar graphs, {\em Eurocomb '07}, September 2007, Seville University.

\bibitem{hoffman}
A.J. Hoffman and R.R. Singleton, On Moore graphs with diameters 2 and 3,
{\em IBM Journal of Research and Development}, 
{\bf 4(5)} November 1960, 497--504. 

\bibitem{huxley}
M.N. Huxley, On the difference between consecutive primes, {\em Invent. Math.} {\bf 15} (1972) 164--170.

\bibitem{juvan}
M. Juvan, B. Mohar, and R. \v{S}krekovski,  List total colourings of graphs,
{\em Combinatorics, Probability and Computing} {\bf 7} (1998), 181--188.

\bibitem{kostochka}
A.V. Kostochka and D.R. Woodall, 
Choosability conjectures and multicircuits.  
{\em Discrete Math.}  {\bf 240}  (2001),  no. 1-3, 123--143. 


\bibitem{miller}
M. Miller and J. \v{S}ir\'{a}\v{n}, Moore graphs and beyond: A survey of the degree/diameter problem, {\em \url{www.combinatorics.org/Surveys/ds14.pdf}}
(2005) p. 12.

\bibitem{MS} M. Molloy and M. Salavatipour, A bound on the
chromatic number of the square of a planar graph, {\em J. Combin.
Theory Ser B.}, {\bf 94} (2005), 189--213.


\bibitem{Ko-Wei} W.-P. Wang and K.-W. Lih,
Labeling planar graphs with conditions on girth and distance two,
{\em SIAM J. Discrete Math.}, {\bf 17} (2003), 264--275.

\bibitem{Thomassen} C. Thomassen, 
The square of a planar cubic graph is 7-colorable, {\em manuscript}.

\bibitem{Wegner} G. Wegner, Graphs with given diameter and a coloring
problem, {\em preprint}, University of Dortmund, Dortmund,
Germany, 1977.

\bibitem{IGT} D.B. West, {\em Introduction to Graph Theory.} 
Prentice Hall, second edition, 2001.

\end{thebibliography}
\end{document}